\documentclass{article}
\usepackage{amsmath}
\usepackage{amssymb}

\newtheorem{theorem}{Theorem}
\newtheorem{proposition}{Proposition}

\newtheorem{lemma}{Lemma}
\newtheorem{remark}{Remark}
\newenvironment{proof}[1][Proof]{\noindent\textbf{#1.} }{\ \rule{0.5em}{0.5em}}
\input{tcilatex}

\begin{document}

\title{Ellipses of minimal eccentricity inscribed in midpoint diagonal
quadrilaterals.}
\author{Alan Horwitz}
\date{9/2/16}
\maketitle

\begin{abstract}
In \cite{H1} we showed that there is a unique ellipse of minimal
eccentricity, $E_{I}$, inscribed in any convex quadrilateral, $Q$. Using a
different approach than in \cite{H1}, we prove that there is a unique
ellipse of minimal eccentricity, $E_{I}$, inscribed in a midpoint diagonal
quadrilateral, $Q$, which is a quadrilateral with the property that the
intersection point of the diagonals of $Q$ coincides with the midpoint of at
least one of the diagonals of $Q$. Our main result is that if $Q$ is a
midpoint diagonal quadrilateral, then the smallest non--negative angle
between equal conjugate diameters of $E_{I}$ equals the smallest
non--negative angle between the diagonals of $Q$. This was proven in \cite%
{H2} for parallelograms.
\end{abstract}

\section{Introduction}

In \cite{H1} the author proved numerous results about ellipses inscribed%
\textbf{\ }in convex quadrilaterals, $Q$. By \textit{inscribed} we mean that
the ellipse lies inside $Q$ and is tangent to each side of $Q$. In
particular, we proved that there exists a unique ellipse, $E_{I}$, of
minimal eccentricity inscribed in $Q$. In \cite{H2} we gave the following
geometric characterization of $E_{I}$ for parallelograms: The smallest
nonnegative angle, $\Gamma $, between equal conjugate diameters of $E_{I}$
equals the smallest nonnegative angle, $\alpha $, between the diagonals of $%
Q $. The main result in this paper(Theorem \ref{T1}) is to extend this
result to a larger class of quadrilaterals we call midpoint diagonal
quadrilaterals. A quadrilateral, $Q$, is called a midpoint diagonal
quadrilateral if the intersection point of the diagonals of $Q$ coincides
with the midpoint of at least one of the diagonals of $Q$. This includes the
possibility that $Q$ is a parallelogram, in which case the diagonals of $Q$
bisect one another. Equivalently, if $Q$ is not a parallelogram, then $Q$ is
a midpoint diagonal quadrilateral if and only if the line thru the midpoints
of the diagonals of $Q$ contains one of the diagonals of $Q$. For convex
quadrilaterals in general, some simple examples show that $\Gamma \neq
\alpha $. In addition, other examples show that there are convex
quadrilaterals which are not midpoint diagonal quadrilaterals with $\Gamma
=\alpha $.

The method of proof in this paper is somewhat different than in \cite{H1},
where we used a theorem of Marden relating the foci of an ellipse tangent to
the lines thru the sides of a triangle and the zeros of a partial fraction
expansion. In this paper we use formulas for the lengths of the major and
minor axes of an ellipse, $E_{0}$, inscribed\textbf{\ }in $Q$, as a function
of the coefficients of an equation of $E_{0}$(Lemma \ref{L4}), and hence we
have a formula for the eccentricity of $E_{0}$ as a function of the
coefficients as well. This approach is shorter and more direct. As noted
above, Theorem \ref{T1} holds for parallelograms. Also, we show below(Lemma %
\ref{L2}) that a midpoint diagonal quadrilateral cannot be a trapezoid. Thus
we may assume throughout this paper(unless stated otherwise) that no two
sides of $Q$ are parallel.

\section{Locus of Centers of Ellipses inscribed in Quadrilaterals}

Before proving our main result, Theorem \ref{T1}, we need more general
results about ellipses inscribed in quadrilaterals which will be useful
later. The following problem, often referred to in the literature as
Newton's problem, is to determine the locus of centers of ellipses inscribed
in $Q$. The solution given by Newton is described in the following
theorem(see \cite{C} and \cite{D}).

\begin{theorem}
(Newton)\label{Newton}: Let $M_{1}$ and $M_{2}$ be the midpoints of the
diagonals of a quadrilateral, $Q$. If $E$ is an ellipse inscribed in $Q$,
then the center of $E$ must lie on the open line segment, $Z$, connecting $%
M_{1}$ and $M_{2}$.
\end{theorem}

\begin{remark}
If $Q$ is a parallelogram, then the diagonals of $Q$ intersect at the
midpoints of the diagonals of $Q$, and thus $Z$ is really just one point.
\end{remark}

The classical proof of Theorem \ref{Newton} involves first using an
orthogonal projection to map $E$ to a circle, $C$, and then proving Theorem %
\ref{Newton} for $C$. Affine invariance then allows one to obtain Theorem %
\ref{Newton} for ellipses in general. However, Theorem \ref{Newton} does not
really give the precise locus of centers of ellipses inscribed in $Q$, but
only shows that the center of $E$ must lie on $Z$. What about the converse ?
That is, is \textbf{every point} of $Z$ the center of some ellipse inscribed
in $Q$ ? The following theorem shows that the locus of centers of ellipses
inscribed in $Q$ is precisely $Z$.

\begin{theorem}
\label{locus of centers}\cite{H1}Let $Q$\ be a convex quadrilateral in the $%
xy$ plane with no two sides parallel. Let $M_{1}$ and $M_{2}$ be the
midpoints of the diagonals of $Q$, and let $Z$ be the open line segment
connecting $M_{1}$ and $M_{2}$. If $(h,k)\in Z$, then there is a unique
ellipse with center $(h,k)$ inscribed in $Q$.
\end{theorem}

Theorem \ref{locus of centers} was proven in \cite{H1} using a result of
Marden relating the foci of an ellipse tangent to the lines thru the sides
of a triangle and the zeros of a partial fraction expansion. We provide a
different proof here and we also fill in a gap in \cite{H1} for the proof of
uniqueness. The following proposition allows us to fill that gap.

\begin{proposition}
\label{P1}Suppose that $E_{1}$ and $E_{2}$ are distinct ellipses with the
same center and which are each inscribed in the same convex quadrilateral, $%
Q $. Then $Q$ must be a parallelogram.
\end{proposition}

\begin{remark}
Chakerian(\cite{C}) mentions the essence of Proposition \ref{P1}, but no
proof is cited or given.
\end{remark}

\begin{proof}
Note that ellipses, tangent lines to ellipses, convex quadrilaterals, and
parallelograms are preserved under nonsingular affine transformations.
First, since $E_{1}$ and $E_{2}$ have the same center, by applying a
translation, one can assume that $E_{1}$ and $E_{2}$ have center $=(0,0)$;
Now use a scaling transformation which maps $E_{1}$ to the unit circle and
thus leaves $E_{2}$ as an ellipse. Finally, by applying a rotation about the
origin, we can assume that $E_{2}$ has major and minor axes parallel to the $%
x$ and $y$ axes, respectively. The equations of $E_{1}$ and $E_{2}$ are then 
$x^{2}+y^{2}=1$ with slope function $\dfrac{dy}{dx}=-\dfrac{x}{y},y\neq 0$,
and $\dfrac{x^{2}}{a^{2}}+\dfrac{y^{2}}{b^{2}}=1$, with slope function $%
\dfrac{dy}{dx}=-\dfrac{b^{2}}{a^{2}}\dfrac{x}{y},y\neq 0$; Note that $a\neq
b $ since $E_{1}\neq E_{2}$ and hence $a>b$; Since $E_{1}$ and $E_{2}$ are
each inscribed in the same convex quadrilateral, there must be four distinct
lines which are tangent to each of the curves $E_{1}$ and $E_{2}$; If $a=1$,
then there are only two lines(the vertical lines $x=\pm 1$) which are
tangent to both $E_{1}$ and $E_{2}$, while if $b=1$, then that there are
only two lines(the horizontal lines $y=\pm 1$) which are tangent to both $%
E_{1}$ and $E_{2}$. Thus we may assume that $a\neq 1\neq b$; It is then
clear that any line tangent to both $E_{1}$ and $E_{2}$ cannot be horizontal
or vertical. So suppose that the line, $L$, is tangent to $E_{1}$ and $E_{2}$
at the points $P_{1}=(x_{1},y_{1})$ and $P_{2}=(x_{2},y_{2})$, respectively.
Then $-1<x_{1}<1$,$-1<y_{1}<1,-a<x_{2}<a$, and $-b<y_{2}<b$ since $L$ is not
vertical. Also, $L$ has equation 
\begin{equation}
y=mx+B,m\neq 0  \label{taneq}
\end{equation}%
since $L$ is not horizontal. Note that $x_{1},y_{1},x_{2}$, and $y_{2}$ are
all nonzero. We then have $x_{1}^{2}+y_{1}^{2}=1,\dfrac{dy}{dx}=-\dfrac{x_{1}%
}{y_{1}}$ if $y_{1}\neq 0$ for $P_{1}$ and $\dfrac{x_{2}^{2}}{a^{2}}+\dfrac{%
y_{2}^{2}}{b^{2}}=1$, $\dfrac{dy}{dx}=-\dfrac{b^{2}}{a^{2}}\dfrac{x_{2}}{%
y_{2}}$ if $y_{2}\neq 0$ for $P_{2}$; Since $L$ is tangent to $E_{1}$ and $%
E_{2}$ at $P_{1}$ and $P_{2}$, respectively, the equation of $L$ is also
given by $y=-\dfrac{x_{1}}{y_{1}}x+y_{1}+\dfrac{x_{1}^{2}}{y_{1}}=-\dfrac{%
x_{1}}{y_{1}}x+\dfrac{1}{y_{1}}$ and by $y=-\dfrac{b^{2}}{a^{2}}\dfrac{x_{2}%
}{y_{2}}x+y_{2}+\dfrac{b^{2}}{a^{2}}\dfrac{x_{2}^{2}}{y_{2}}=-\dfrac{b^{2}}{%
a^{2}}\dfrac{x_{2}}{y_{2}}x+\dfrac{b^{2}}{y_{2}}$; Hence the following
system of equations holds: 
\begin{gather}
-\dfrac{x_{1}}{y_{1}}=-\dfrac{b^{2}}{a^{2}}\dfrac{x_{2}}{y_{2}}=m  \label{3}
\\
\dfrac{1}{y_{1}}=\dfrac{b^{2}}{y_{2}}=B\text{.}  \label{4}
\end{gather}%
Using (\ref{3}) we have $m^{2}=\dfrac{x_{1}^{2}}{y_{1}^{2}}=\dfrac{x_{1}^{2}%
}{1-x_{1}^{2}}$, and $m^{2}=\dfrac{b^{4}}{a^{4}}\dfrac{x_{2}^{2}}{y_{2}^{2}}=%
\dfrac{b^{4}}{a^{4}}\dfrac{x_{2}^{2}}{\tfrac{b^{2}}{a^{2}}(a^{2}-x_{2}^{2})}%
=\allowbreak \dfrac{b^{2}}{a^{2}}\dfrac{x_{2}^{2}}{a^{2}-x_{2}^{2}}$, which
implies that $x_{2}^{2}=\dfrac{a^{4}m^{2}}{a^{2}m^{2}+b^{2}}=\dfrac{a^{4}%
\tfrac{x_{1}^{2}}{1-x_{1}^{2}}}{a^{2}\tfrac{x_{1}^{2}}{1-x_{1}^{2}}+b^{2}}$,
which implies that 
\begin{equation}
x_{2}^{2}=\dfrac{a^{4}x_{1}^{2}}{(a^{2}-b^{2})x_{1}^{2}+b^{2}}\text{.}
\label{6}
\end{equation}%
Now (\ref{3})$\ $also implies that $\dfrac{x_{1}}{y_{1}}=\dfrac{b^{2}}{a^{2}}%
\dfrac{x_{2}}{y_{2}}$, and (\ref{4}) yields $\dfrac{x_{1}}{y_{1}}%
=x_{1}\left( \dfrac{b^{2}}{y_{2}}\right) $, and hence $x_{1}\left( \dfrac{%
b^{2}}{y_{2}}\right) =\dfrac{b^{2}}{a^{2}}\dfrac{x_{2}}{y_{2}}$, which
implies that $\dfrac{x_{2}}{x_{1}}=a^{2}$; (\ref{4}) also yields $\dfrac{%
y_{2}}{y_{1}}=b^{2}$; Hence $x_{1}$ and $x_{2}$ must have the same sign, and 
$y_{1}$ and $y_{2}$ also must have the same sign; $\dfrac{x_{2}^{2}}{%
x_{1}^{2}}=a^{4}$ and (\ref{6}) gives $(a^{2}-b^{2})x_{1}^{2}+b^{2}=1$,
which implies that 
\begin{equation}
x_{1}^{2}=\dfrac{1-b^{2}}{a^{2}-b^{2}}\text{.}  \label{5}
\end{equation}%
(\ref{5}) and $x_{2}^{2}=a^{4}x_{1}^{2}$ then implies that 
\begin{equation}
x_{2}^{2}=a^{4}\left( \dfrac{1-b^{2}}{a^{2}-b^{2}}\right) \text{.}  \label{7}
\end{equation}

Now it follows easily that if $a<1$(and thus $b<1$ since $b<a$), then $E_{2}$%
\ is contained in $E_{1}$, which would imply that $E_{1}$ and $E_{2}$ cannot
each be inscribed in the same convex quadrilateral. Similarly, if $b>1$(and
thus $a>1$ since $a>b$), then $E_{1}$\ is contained in $E_{2}$, and again $%
E_{1}$ and $E_{2}$ could not each be inscribed in the same convex
quadrilateral. Thus $a\geq 1$ and $b\geq 1$, and since we assumed above that 
$a\neq 1\neq b$, we must have $b<1<a$; (\ref{5}) then yields $x_{1}=\pm 
\sqrt{\tfrac{1-b^{2}}{a^{2}-b^{2}}}$, and for each choice of $\pm $ sign for 
$x_{1}$, one has $y_{1}=\pm \sqrt{1-x_{1}^{2}}=\pm \sqrt{\tfrac{a^{2}-1}{%
a^{2}-b^{2}}}$. That yields four distinct points $Q_{j}=(x_{1},y_{1})=\left(
\pm \sqrt{\tfrac{1-b^{2}}{a^{2}-b^{2}}},\pm \sqrt{\tfrac{a^{2}-1}{a^{2}-b^{2}%
}}\right) ,j=1,2,3,4$; Define the following four lines $y=mx+B$, where $m=-%
\dfrac{x_{1}}{y_{1}}$ and $B=\dfrac{1}{y_{1}}$ for each choice of $%
(x_{1},y_{1})$ above: $L_{1}$: $y=-\sqrt{\tfrac{1-b^{2}}{a^{2}-1}}x-\sqrt{%
\tfrac{a^{2}-b^{2}}{a^{2}-1}}$, $L_{2}$: $y=\sqrt{\tfrac{1-b^{2}}{a^{2}-1}}x+%
\sqrt{\tfrac{a^{2}-b^{2}}{a^{2}-1}}$,

$L_{3}$: $y=\sqrt{\tfrac{1-b^{2}}{a^{2}-1}}x-\sqrt{\tfrac{a^{2}-b^{2}}{%
a^{2}-1}}$, and $L_{4}$: $y=-\sqrt{\tfrac{1-b^{2}}{a^{2}-1}}x+\sqrt{\tfrac{%
a^{2}-b^{2}}{a^{2}-1}}$; Then it follows immediately that $L_{1},L_{2},L_{3}$%
, and $L_{4}$ are tangent to $E_{1}$ at the $Q_{j}$\ since $m=-\dfrac{x_{1}}{%
y_{1}}$ and $B=\dfrac{1}{y_{1}}$; (\ref{7}) then yields $x_{2}=\pm a^{2}%
\sqrt{\tfrac{1-b^{2}}{a^{2}-b^{2}}}$, where the $+$ or $-$ sign are chosen
so that $x_{1}$ and $x_{2}$ have the same sign.Then $y_{2}=\pm \dfrac{b}{a}%
\sqrt{a^{2}-x_{2}^{2}}=\pm \dfrac{b}{a}\sqrt{\tfrac{a^{2}b^{2}\left(
a^{2}-1\right) }{a^{2}-b^{2}}}=\pm b^{2}\sqrt{\tfrac{a^{2}-1}{a^{2}-b^{2}}}$%
, where again the $+$ or $-$ sign are chosen so that $y_{1}$ and $y_{2}$
have the same sign. Since $\dfrac{x_{1}}{y_{1}}=\dfrac{b^{2}}{a^{2}}\dfrac{%
x_{2}}{y_{2}}$ and $\dfrac{1}{y_{1}}=\dfrac{b^{2}}{y_{2}}$, it follows that $%
m=-\dfrac{b^{2}}{a^{2}}\dfrac{x_{2}}{y_{2}}$ and $B=\dfrac{b^{2}}{y_{2}}$,
which implies that $L_{1},L_{2},L_{3}$, and $L_{4}$ are also tangent to $%
E_{2}$ at the four distinct points $R_{j}=\left( \pm a^{2}\sqrt{\tfrac{%
1-b^{2}}{a^{2}-b^{2}}},\pm b^{2}\sqrt{\tfrac{a^{2}-1}{a^{2}-b^{2}}}\right)
,j=1,2,3,4$; Now $L_{1}\parallel L_{4}$ and $L_{2}\parallel L_{3}$, which
implies that $L_{1},L_{2},L_{3}$, and $L_{4}\ $must form a parallelogram.
Furthermore, $L_{1},L_{2},L_{3}$, and $L_{4}$ are the only lines which are
common tangents to $E_{1}$ and $E_{2}$ since we have shown that $x_{1}=\pm 
\sqrt{\tfrac{1-b^{2}}{a^{2}-b^{2}}}$ and $x_{2}=\pm a^{2}\sqrt{\tfrac{1-b^{2}%
}{a^{2}-b^{2}}}$ are the only solutions of (\ref{3}) and (\ref{4}). Thus $%
E_{1}$ and $E_{2}$ are each inscribed in the same convex quadrilateral, $Q$,
and $Q$ must be a parallelogram.

Now, to prove Theorem \ref{locus of centers}, one can use any nonsingular
affine transformation to map $Q$ to a quadrilateral of a simpler form.
However, this will not work to prove Theorem \ref{T1} since the ratio of the
eccentricity of two ellipses is \textbf{not} preserved in general under
nonsingular affine transformations of the plane. For brevity, we want to use
the same quadrilateral for the rest of this paper. So, by using an \textbf{%
isometry} of the plane, we can assume that $Q$\ has vertices $%
(0,0),(0,u),(s,t)$, and $(v,w)$, where
\end{proof}

\begin{equation}
s>0,v>0,u>0,t>w\text{.}  \label{R0}
\end{equation}%
The sides of $Q$, going clockwise, are given by $S_{1}=\overline{(0,0)\ (v,w)%
},S_{2}=\overline{(0,0)\ (0,u)},$

$S_{3}=\overline{(0,u)\ (s,t)}$, and $S_{4}=\overline{(s,t)\ (v,w)}$; Let $%
L_{1}$: $y=\dfrac{w}{v}x,L_{2}$: $x=0,L_{3}$: $y=u+\dfrac{t-u}{s}x$, and $%
L_{4}$: $y=w+\dfrac{t-w}{s-v}(x-v)$ denote the lines which make up the
boundary of $Q$.

\textbullet\ Since $Q$\ is convex, $(s,t)$ must lie above $%
\overleftrightarrow{(0,u)\ (v,w)}$\ and $(v,w)$ must lie below $%
\overleftrightarrow{(0,0)\ (s,t)}$, which implies 
\begin{equation}
v(t-u)+(u-w)s>0,vt-ws>0\text{.}  \label{R1}
\end{equation}

\textbullet\ Since no two sides of $Q$\ are parallel, $L_{1}\nparallel L_{3}$
and $L_{2}\nparallel L_{4}$, which implies

\begin{equation}
ws-v(t-u)\neq 0,s\neq v\text{.}  \label{R2}
\end{equation}
Assume now that (\ref{R0}), (\ref{R1}), and (\ref{R2}) hold throughout this
section. Let

\begin{equation}
I=\left\{ 
\begin{array}{ll}
(v/2,s/2) & \text{if }v<s \\ 
(s/2,v/2) & \text{if }s<v%
\end{array}%
\right. .  \label{I}
\end{equation}

Note that 
\begin{eqnarray}
(s-2h)(2h-v) &>&0,h\in I,  \label{10} \\
(s-2h)(s-v) &>&0,h\in I,  \label{11} \\
(2h-v)(s-v) &>&0,h\in I\text{.}  \label{12}
\end{eqnarray}%
$M_{1}=\left( \dfrac{1}{2}v,\dfrac{1}{2}(w+u)\right) $ and $M_{2}=\left( 
\dfrac{1}{2}s,\dfrac{1}{2}t\right) $ are the midpoints of the diagonals of $%
Q $\ and the equation of the line thru $M_{1}$ and $M_{2}$ is

\begin{equation}
y=L(x)=\dfrac{t}{2}+\dfrac{w+u-t}{v-s}\left( x-\dfrac{s}{2}\right) ,x\in I%
\text{.}  \label{9}
\end{equation}%
The diagonals of $Q$ are $D_{1}=\overline{(0,0)\ (s,t)}=$ diagonal from
lower left to upper right and $D_{2}=\overline{(0,u)\ (v,w)}=$ diagonal from
lower right to upper left.

\begin{lemma}
\label{L3}Define the following linear functions of $h$: $L_{1}(h)=2{\large (}%
v(t-u)-ws{\large )}h+v{\large (}s(u+w)-vt{\large )},L_{2}(h)=2{\large (}%
v(u-t)+ws{\large )}h+s{\large (}v(t-2u)+s(u-w){\large )}$, $L_{3}(h)={\large %
(}v(t-u)+(u-w)s{\large )}(s-2h)$, $L_{4}(h)=-2uh+vt+s(u-w)$, and $%
L_{5}(h)=2\allowbreak {\large (}v(t-u)-ws{\large )}h+uvs$; Then $\left(
s-v\right) L_{j}(h)>0$ on $I,j=1,2,3$, and $L_{j}(h)>0$ on $I,j=4,5$.
\end{lemma}

\begin{proof}
Since each $L_{j}$ is a linear function of $h$, it suffices to prove that $%
\left( s-v\right) L_{j}$ is non--negative at the endpoints of $I,j=1,2,3$,
and that $L_{j}$ is non--negative at the endpoints of $I,j=4,5$. We have $%
L_{1}\left( \dfrac{v}{2}\right) =\allowbreak uv\left( s-v\right) $ and $%
L_{1}\left( \dfrac{s}{2}\right) =\left( s-v\right) \left( vt-ws\right) $, $%
L_{2}\left( \dfrac{v}{2}\right) =\allowbreak \left( s-v\right) {\large (}%
v(t-u)+(u-w)s{\large )}$ and $L_{2}\left( \dfrac{s}{2}\right) =\allowbreak
su\left( s-v\right) $, and

$L_{3}\left( \dfrac{v}{2}\right) =\left( s-v\right) {\large (}v(t-u)+(u-w)s%
{\large )}$ and $L_{3}\left( \dfrac{s}{2}\right) =0$; By (\ref{R0}) and (\ref%
{R1}), $\left( s-v\right) L_{j}(h)>0$ on $I,j=1,2,3$; Also,

$L_{4}\left( \dfrac{v}{2}\right) =\allowbreak v(t-u)+(u-w)s$ and $%
L_{4}\left( \dfrac{s}{2}\right) =vt-ws$, which implies that $L_{4}(h)>0$ on $%
I$ by (\ref{R1}). Finally, $L_{5}\left( \dfrac{v}{2}\right) =v{\large (}%
v(t-u)+(u-w)s{\large )}$ and

$L_{5}\left( \dfrac{s}{2}\right) =s(vt-ws)$, which implies that $L_{5}(h)>0$
on $I$, by (\ref{R0}) and (\ref{R1}).
\end{proof}

Now define the following cubic polynomial, where $L_{5}$ is given by Lemma %
\ref{L3}. 
\begin{equation}
R(h)=(s-2h)(2h-v)L_{5}(h)\text{.}  \label{R3}
\end{equation}

\ Note that the three roots of $R$ are $h_{1}=\dfrac{1}{2}v,h_{2}=\dfrac{1}{2%
}s$, and $h_{3}=\dfrac{1}{2}\dfrac{uvs}{ws-v(t-u)}$; By (\ref{10}) and Lemma %
\ref{L3}, 
\begin{equation}
R(h)>0\ \text{on }I\text{.}  \label{8}
\end{equation}

We now state a result, without proof, about when a quadratic equation in $x$
and $y$ yields an ellipse. The first condition ensures that the conic is an
ellipse, while the second condition ensures that the conic is nondegenerate.

\begin{lemma}
\label{L1}The equation $Ax^{2}+Bxy+Cy^{2}+Dx+Ey+F=0$, with $A,C>0$, is the
equation of an ellipse if and only if $4AC-B^{2}>0$ and $%
CD^{2}+AE^{2}-BDE-F(4AC-B^{2})>0$.
\end{lemma}

The following proposition gives the equation of an ellipse inscribed in the
quadrilateral $Q$ with vertices $(0,0),(0,u),(v,w)$, and $(s,t)$. We
obtained (\ref{1}) below using some results from \cite{H1} along with a
method for obtaining the equation of an ellipse given its foci and the
lengths of the axes of the ellipse. We do not provide those somewhat
cumbersome details here. This equation was not given in \cite{H1}. Rather,
we state the equation below and prove that the equation works using
elementary calculus. One could, of course, attempt to derive the equation of
an ellipse inscribed in the quadrilateral $Q$ with vertices $%
(0,0),(0,u),(v,w)$, and $(s,t)$ by solving a system of nonlinear equations
for the unknown coefficients, $A$ thru $F$, of the equation, and for the
unknown points of tangency. We tried this for a simpler quadrilateral and it
is somewhat messy.

\begin{proposition}
\label{P2}Let $Q$ be the quadrilateral with vertices $(0,0),(0,u),(v,w)$,
and $(s,t)$, which satisfy (\ref{R0}), (\ref{R1}), and (\ref{R2}). Let $I$
be given by (\ref{I}) and let $L$ be given by (\ref{9}). Then $E_{0}$ is an
ellipse inscribed in $Q$\ if and only if the general equation of $E_{0}$ is
given by 
\begin{gather}
4\left( s-v\right) {\large (}(s-v)L^{2}(h)+uw(2h-s){\large )}\left(
x-h\right) ^{2}  \notag \\
+4\left( s-v\right) ^{2}h^{2}{\large (}y-L(h){\large )}^{2}+4\left(
s-v\right) \times  \label{1} \\
{\large (}\left( -2t+2u+2w\right) h^{2}+  \notag \\
\left( vt-su-ws-2vu\right) h+\allowbreak svu{\large )}\left( x-h\right) 
{\large (}y-L(h){\large )}  \notag \\
=uR(h)0,h\in I\text{.}  \notag
\end{gather}
\end{proposition}

\begin{proof}
Suppose that the general equation of $E_{0}$ is given by (\ref{1}) for fixed 
$h\in I$; Then we can write the equation of $E_{0}$ as $\Psi (x,y)=0$, where 
$\Psi (x,y)=A(h)x^{2}+B(h)xy+C(h)y^{2}+D(h)x+E(h)y+F(h)$, and 
\begin{eqnarray}
A(h) &=&4\left( s-v\right) ^{2}\left( \left( \dfrac{1}{2}t+\dfrac{w+u-t}{v-s}%
\left( h-\dfrac{1}{2}s\right) \right) ^{2}+\dfrac{wu(2h-s)}{s-v}\right) , 
\notag \\
B(h) &=&4\left( s-v\right) {\large (}2\left( u+w-t\right) h^{2}+{\large (}%
v(t-2u)-s(u+w){\large )}h+uv\allowbreak s{\large ),}  \notag \\
C(h) &=&4\left( s-v\right) ^{2}h^{2},  \label{2} \\
D(h) &=&\allowbreak 2u\left( 2h-s\right) {\large (}2{\large (}v(w+t-u)-2ws%
{\large )}h+v{\large (}s(u+w)-vt{\large )),}  \notag \\
E(h) &=&4uv\left( s-v\right) h\left( 2h-s\right) \text{, and}  \notag \\
F(h) &=&u^{2}\allowbreak v^{2}\left( 2h-s\right) ^{2}\text{.}  \notag
\end{eqnarray}%
First we want to show that $\Psi (x,y)=0$ defines the equation of an
ellipse. Substituting for $A(h)$ thru $F(h)$ using (\ref{2}) and simplifying
gives 
\begin{gather}
4A(h)C(h)-B^{2}(h)=16u\left( s-v\right) ^{2}R(h)\allowbreak {\large ,} 
\notag \\
{\large (}C(h)D^{2}(h)+A(h)E^{2}(h)-B(h)D(h)E(h){\large )}-F(h){\large (}%
4A(h)C(h)-B^{2}(h){\large )}=  \label{29} \\
=\allowbreak 16\left( s-v\right) ^{2}u^{2}R^{2}(h)\text{.}  \notag
\end{gather}%
By (\ref{8}) and Lemma \ref{L1}, (\ref{1}) defines the equation of an
ellipse for any $h\in I$. Now let the $L_{j}$'s be given as in Lemma \ref{L3}%
. Define the following points, which we will show shortly are the points of
tangency of $E_{0}$ with $Q$:

\textbullet\ $\zeta _{1}=\lambda _{1}(v,w)+(1-\lambda _{1})(0,0)$, where $%
\lambda _{1}=\dfrac{\allowbreak \left( s-2h\right) uv}{L_{1}(h)}$, which
implies that $1-\lambda _{1}=\dfrac{\left( vt-ws\right) \allowbreak \left(
2h-v\right) }{L_{1}(h)}$.

\textbullet\ $\zeta _{2}=\lambda _{2}(0,u)+(1-\lambda _{2})(0,0)$, where $%
\lambda _{2}=\dfrac{\left( s-2h\right) v}{2h\left( s-v\right) }$, which
implies that $1-\lambda _{2}=\allowbreak \dfrac{s(2h-v)}{2h\left( s-v\right) 
}$.

\textbullet\ $\zeta _{3}=\lambda _{3}(s,t)+(1-\lambda _{3})(0,u)$, where $%
\lambda _{3}=\dfrac{\left( 2h-v\right) su}{L_{2}(h)}$, which implies that $%
1-\lambda _{3}=\dfrac{{\large (}v(t-u)+(u-w)s{\large )}\left( s-2h\right) }{%
L_{2}(h)}$.

\textbullet\ $\zeta _{4}=\lambda _{4}(v,w)+(1-\lambda _{4})(s,t)$, where $%
\lambda _{4}=\dfrac{L_{3}(h)}{\left( s-v\right) L_{4}(h)}$, which implies
that $1-\lambda _{4}=\dfrac{(vt-ws)\allowbreak \left( 2h-v\right) }{\left(
s-v\right) L_{4}(h)}$.

Then $0<\zeta _{j}<1,j=1,2,3,4$ by Lemma \ref{L3} and (\ref{10})--(\ref{12}%
); Hence $\zeta _{j}\in S_{j},j=1,2,3,4$, where $S_{j}$ denote the sides of $%
Q$ going clockwise with $S_{1}=\overline{(0,0)\ (v,w)}$; It follows easily
that $\Psi (\zeta _{j})=0,j=1,2,3,4$, which implies that $\zeta _{j}\in
E_{0},j=1,2,3,4$; For fixed $h\in I$, let $Z(x,y)=-\dfrac{\partial \Psi
/\partial x}{\partial \Psi /\partial y}$, which represents the slope, $%
\dfrac{dy}{dx}$, of the ellipse. Then $Z(\zeta _{1})=\allowbreak \dfrac{w}{v}%
=$ slope of $L_{1},Z(\zeta _{3})=\allowbreak \dfrac{t-u}{s}=$ slope of $%
L_{3} $, and $Z(\zeta _{4})=\dfrac{t-w}{s-v}=$ slope of $L_{4}$; Hence $%
E_{0} $ is tangent to $S_{1},S_{3}$, and $S_{4}$ at the points $\zeta
_{1},\zeta _{3}$, and $\zeta _{4}$ respectively. Also, $\dfrac{\partial \Psi 
}{\partial y}(\zeta _{2})=0$ and $\dfrac{\partial \Psi }{\partial x}(\zeta
_{2})=\allowbreak -\dfrac{2u}{h}\left( 2h-v\right) \left( s-2h\right)
L_{5}(h)\allowbreak \neq 0$ by Lemma \ref{L3} and (\ref{10}), which implies
that $E_{0}$ is tangent to the vertical line segment $S_{2}$ at $\zeta _{2}$%
. For any simple closed convex curve, such as an ellipse, tangent to each
side of $Q$ then implies that that curve lies in $Q$. That proves that $%
E_{0} $ is inscribed in $Q$. Second, suppose that $E_{0}$ is an ellipse
inscribed in $Q$. By Theorem \ref{Newton},\textbf{\ }$E_{0}$ has center $%
{\large (}h_{1},L(h_{1}){\large )}$ for some $h_{1}\in I$. We have just
shown that (\ref{1}) represents a family of ellipses inscribed in $Q$\ as $h$
varies over $I$, and it is not hard to show that each ellipse given by (\ref%
{1}) has center ${\large (}h,L(h){\large )}$ for some $h\in I$. Let $\tilde{E%
}_{0} $ be the ellipse given by (\ref{1}) with $h=h_{1}$. Hence $\tilde{E}%
_{0}$ also has center ${\large (}h_{1},L(h_{1}){\large )}$ and is inscribed
in $Q$. Since $(0,0),(0,u),(v,w)$, and $(s,t)$ satisfy (\ref{R2}), $Q$ is
not a parallelogram. Then by Proposition \ref{P1}, $\tilde{E}_{0}=E_{0}$ and
the general equation of $E_{0}$ must be given by (\ref{1}).
\end{proof}

\begin{proof}
(proof of Theorem \ref{locus of centers}): \textit{Existence} follows
immediately from Proposition \ref{P2}. We already proved \textit{uniqueness }%
above in the proof of Proposition \ref{P2}.
\end{proof}

\section{Preliminary Lemmas}

\begin{lemma}
\label{L5}Let $G(h)=\dfrac{J(h)-\sqrt{M(h)}}{J(h)+\sqrt{M(h)}}$, where $J$
and $M$ are differentiable functions on some interval, $I$, with $J(h)+\sqrt{%
M(h)}\neq 0$ on $I$ and $M(h)>0$ on $I$. Then $G$ is differentiable on $I$
and $G^{\prime }(h)=\dfrac{2J^{\prime }(h)M(h)-J(h)M^{\prime }(h)}{\sqrt{M(h)%
}{\large (}J(h)+\sqrt{M(h)}{\large )}^{2}}$ for $h\in I$.
\end{lemma}

\begin{proof}
It is clear that $G$ is differentiable on $I$, and the rest of the lemma
follows from the quotient rule after some simplification.
\end{proof}

Before proving Theorem \ref{T1}, we state and prove the following series of
lemmas. The purpose of these lemmas will be to show that the eccentricity of
an inscribed ellipose, as a function of $h$, has a unique root in $I$, where 
$I$ is given by (\ref{I}) throughout this section. We also want to find a
formula for that root as well. Assume throughout that (\ref{R0}), (\ref{R1}%
), and (\ref{R2}) hold.

First we define the following quadratic polynomial in $h$: 
\begin{eqnarray}
o(h) &=&-2\left( s^{2}+t^{2}\right) \left( s-v\right) h^{2}-2Kh+sK,
\label{oK} \\
K &=&(s^{2}+t^{2})v^{2}-2ws(vt-ws)\text{.}  \notag
\end{eqnarray}

\begin{lemma}
\label{L6}: $K>0$.
\end{lemma}

\begin{proof}
Writing $K$ from (\ref{oK}) as a quadratic in $v$, $K(v)=\allowbreak \left(
s^{2}+t^{2}\right) v^{2}-2wstv+2s^{2}w^{2}$, the discriminant of $K(v)$ is $%
(-2wst)^{2}-4\left( s^{2}+t^{2}\right) (2s^{2}w^{2})=\allowbreak
-4s^{2}w^{2}\left( t^{2}+2s^{2}\right) <0$, which implies that $K(v)$ has no
real roots. Since $K(0)>0$, it follows that $K(v)>0$, for all real $v$.
\end{proof}

\begin{lemma}
\label{L7}Let $p_{1}=v^{2}(s^{2}+t^{2})-4ws\left( vt-ws\right) $; Then $%
p_{1}>0$.
\end{lemma}

\begin{proof}
Rewrite $p_{1}=(2ws-vt)^{2}+v^{2}s^{2}>0$ since $s,v\neq 0$.
\end{proof}

\begin{lemma}
\label{L8}$o$ has exactly one root in $I$.
\end{lemma}

\begin{proof}
\begin{equation}
o\left( \dfrac{v}{2}\right) =\allowbreak \dfrac{1}{2}\left( s-v\right)
p_{1}\allowbreak ,o\left( \dfrac{s}{2}\right) =-\allowbreak \dfrac{1}{2}%
s^{2}\left( s-v\right) \left( s^{2}+t^{2}\right) \text{.}  \label{oendpts}
\end{equation}%
By Lemma \ref{L7}, $o\left( \dfrac{v}{2}\right) o\left( \dfrac{s}{2}\right)
<0$, which implies that $o$ has an odd number of roots in $I$; Since $o$ is
a quadratic, $o$ must have one root in $I$. Note that since $o$ has two
distinct real roots, the discriminant of $o,4K^{2}+8\left(
s^{2}+t^{2}\right) \left( s-v\right) sK=4K{\large (}2\left(
s^{2}+t^{2}\right) s\left( s-v\right) +K{\large )}$, is positive. By Lemma %
\ref{L6}, $2\left( s^{2}+t^{2}\right) s\left( s-v\right) +K>0$\ and by the
quadratic formula, the roots of $o$ are
\end{proof}

\begin{eqnarray}
h_{\_} &=&\sqrt{K}\dfrac{-\sqrt{K}-\sqrt{2\left( s^{2}+t^{2}\right) s\left(
s-v\right) +K}}{2\left( s^{2}+t^{2}\right) \left( s-v\right) },
\label{oroots} \\
h_{+} &=&\sqrt{K}\dfrac{-\sqrt{K}+\sqrt{2\left( s^{2}+t^{2}\right) s\left(
s-v\right) +K}}{2\left( s^{2}+t^{2}\right) \left( s-v\right) }\text{.} 
\notag
\end{eqnarray}

Note that $h_{+}-h_{\_}=\dfrac{\sqrt{K}\sqrt{2\left( s^{2}+t^{2}\right)
s\left( s-v\right) +K}}{\left( s^{2}+t^{2}\right) \left( s-v\right) }$,
which implies that $\left\{ 
\begin{array}{ll}
h_{\_}>h_{+} & \text{if }s<v \\ 
h_{+}>h_{\_} & \text{if }v<s%
\end{array}%
\right. $.

The following lemma allows us to find the unique root of $o$ in $I$.

\begin{lemma}
\label{L10}$h_{+}$ is the unique root of $o$ in $I$.
\end{lemma}

\begin{proof}
\textbf{Case 1:} $s>v$; Then $I=\left( \dfrac{v}{2},\dfrac{s}{2}\right) $
and $h_{\_}<0$, which implies that $h_{\_}\notin I$; By Lemma \ref{L8}, $%
h_{+}\in I$.

\textbf{Case 2:} $s<v$; Then $I=\left( \dfrac{s}{2},\dfrac{v}{2}\right) $;
Since $\lim\limits_{h\rightarrow \infty }o(h)=\infty $ and $o\left( \dfrac{v%
}{2}\right) <0$ by (\ref{oroots}) and Lemma \ref{L7}, $o$ has a root in the
interval $\left( \dfrac{v}{2},\infty \right) $; Hence the smaller of the two
roots of $o$ must lie in $I$; Since $h_{+}<h_{\_}$, one must have $h_{+}\in
I $.
\end{proof}

Now define the polynomials 
\begin{eqnarray}
J(h) &=&A(h)+C(h),  \label{jm} \\
M(h) &=&{\large (}A(h)-C(h){\large )}^{2}+{\large (}B(h){\large )}^{2}\text{,%
}  \notag
\end{eqnarray}

where $A(h),B(h)$, and $C(h)$ are given by (\ref{2}). Note that $M(h)\geq 0$%
; Some simplification yields%
\begin{equation}
J^{2}(h)-M(h)=16u\left( s-v\right) ^{2}R(h)\text{,}  \label{jmr}
\end{equation}

where $R$ is given by (\ref{R3}).

\begin{lemma}
\label{L11}$J(h)>0,h\in I$.
\end{lemma}

\begin{proof}
If $J(h_{0})=0$ for some $h_{0}\in I$, then $R(h_{0})\leq 0$ by (\ref{jmr}),
which contradicts (\ref{8}); Since $J\left( \dfrac{v}{2}\right) =\allowbreak
\left( s-v\right) ^{2}{\large (}(w-u)^{2}+v^{2}{\large )}>0$, it follows
that $J(h)>0,h\in I$.
\end{proof}

\begin{lemma}
\label{L14}Suppose that $Q$\ has vertices $(0,0),(0,u),(s,t)$, and $(v,w)$,
where $s,t,u,v,w$ satisfy (\ref{R0}), (\ref{R1}), and (\ref{R2}).

(i) $Q$ is a type 1 midpoint diagonal quadrilateral if and only if 
\begin{equation}
u=\allowbreak \dfrac{vt-ws}{s}\text{.}  \label{u1}
\end{equation}

(ii) $Q$ is a type 2 midpoint diagonal quadrilateral if and only if 
\begin{equation}
u=\dfrac{vt-ws}{2v-s}\text{.}  \label{u2}
\end{equation}
\end{lemma}

\begin{proof}
Recall that $L(x)=\dfrac{t}{2}+\dfrac{w+u-t}{v-s}\left( x-\dfrac{s}{2}%
\right) $. The line containing $D_{1}$ is $y=\dfrac{t}{s}x$; Then $%
D_{1}=L\iff \dfrac{w+u-t}{v-s}=\dfrac{t}{s}$ and $\dfrac{t}{2}-\dfrac{s}{2}%
\dfrac{w+u-t}{v-s}=0\iff $ (\ref{u1}) holds. That proves (i). The line
containing $D_{2}$ is $y=u+\dfrac{w-u}{v}x$; Then $D_{2}=L\iff \dfrac{w+u-t}{%
v-s}=\dfrac{w-u}{v}$ and $\dfrac{t}{2}-\dfrac{s}{2}\dfrac{w+u-t}{v-s}=u\iff
(\allowbreak 2v-s)u+ws-vt=0$; Note that if $2v-s=0$, then $(\allowbreak
2v-s)u+ws-vt\neq 0$ since $ws-vt\neq 0$ by (\ref{R1}); Thus $Q$ is a type 2
midpoint diagonal quadrilateral if and only if (\ref{u2}) holds. That proves
(ii).
\end{proof}

We now prove two results about midpoint diagonal quadrilaterals. We define a
trapezoid to be a quadrilateral with at least one pair of parallel sides.

\begin{lemma}
\label{L2}Suppose that $Q$\ is a midpoint diagonal quadrilateral which is
also a trapezoid. Then $Q$ is a parallelogram.
\end{lemma}

\begin{proof}
Suppose that $Q$\ is a midpoint diagonal quadrilateral which is a trapezoid,
but which is not a parallelogram. By affine invariance, we may assume that $%
Q $ is the trapezoid with vertices $(0,0),(1,0),(0,1)$, and $(1,t),0<t\neq 1$%
; The diagonals of $Q$ are $D_{1}$: $y=tx$ and $D_{2}$: $y=1-x$, and the
midpoints of the diagonals are $M_{1}=\left( \dfrac{1}{2},\dfrac{1}{2}%
\right) $ and $M_{2}=\left( \dfrac{1}{2},\dfrac{1}{2}t\right) $; The open
line segment, $L$, joining $M_{1}$ and $M_{2}$ is contained in the vertical
line $x=\dfrac{1}{2}$; Since the diagonals of $Q$ are nonvertical lines, $%
D_{1}\neq L$ and $D_{2}\neq L$, which implies that $Q$\ is not a midpoint
diagonal quadrilateral.
\end{proof}

Let $D_{1}$ denote the diagonal of $Q$ from lower left to upper right and
let $D_{2}$ denote the diagonal of $Q$ from lower right to upper left. We
note here that there are two types of midpoint diagonal quadrilaterals: Type
1, where $D_{1}=L$ and Type 2, where $D_{2}=L$.

\begin{lemma}
\label{L13}Suppose that $Q$ is both a tangential and a midpoint diagonal
quadrilateral. Then $Q$ is an orthodiagonal quadrilateral.
\end{lemma}

\begin{proof}
Since an isometry preserves both tangential and midpoint diagonal
quadrilaterals(a general affine transformation does not suffice), we can
assume that $Q$\ has vertices $(0,0),(0,u),(s,t)$, and $(v,w)$, where $%
s,t,u,v,w$ satisfy (\ref{R0}), (\ref{R1}), and (\ref{R2}). Now $Q$ is
tangential $\iff Z=0$, where 
\begin{eqnarray}
Z &=&\left( \left( v^{2}+w^{2}\right) \left( s^{2}+(t-u)^{2}\right) -\left(
tu-vs-wt\right) ^{2}-u^{2}\left( (s-v)^{2}+(t-w)^{2}\right) \right) ^{2} 
\notag \\
&&-4{\large (}u\left( tu-vs-wt\right) {\large )}^{2}{\large (}%
(s-v)^{2}+(t-w)^{2}{\large )}\text{.}  \label{Z}
\end{eqnarray}

\textbf{Case 1: }$Q$ is a type 1 midpoint diagonal quadrilateral.
Substituting for $u$ in (\ref{Z}) using (\ref{u1}) yields 
\begin{equation*}
Z=\dfrac{-4\left( s-v\right) ^{2}\left( vt-ws\right) ^{2}{\large (}%
v(t^{2}-s^{2})-2wst{\large )}^{2}}{s^{4}}=0
\end{equation*}%
$\iff $%
\begin{equation}
v(t^{2}-s^{2})-2wst=0\text{.}  \label{27}
\end{equation}%
So if $Q$ is both a tangential and a type 1 midpoint diagonal quadrilateral,
then (\ref{27}) holds. The slopes of the diagonals are $m_{1}=\dfrac{t}{s}$
and $m_{2}=\dfrac{w-u}{v}=\allowbreak \dfrac{2ws-vt}{vs}$, which implies
that $m_{1}m_{2}+1=\dfrac{t}{s}\dfrac{2ws-vt}{vs}+1=\allowbreak \dfrac{%
2wts-v(t^{2}-s^{2})}{s^{2}v}=0$ by (\ref{27}).

\textbf{Case 2: }$Q$ is a type 2 midpoint diagonal quadrilateral.
Substituting for $u$ in (\ref{Z}) using (\ref{u2}) yields 
\begin{equation*}
Z=\dfrac{-4v^{2}\left( s-v\right) ^{2}\left( vt-ws\right) ^{2}{\large (}%
2(vs+wt)-(s^{2}+t^{2}){\large )}^{2}}{\left( 2v-s\right) ^{4}}=0
\end{equation*}%
if and only if 
\begin{equation}
2(vs+wt)-(s^{2}+t^{2})=0\text{.}  \label{36}
\end{equation}

So if $Q$ is both a tangential and a type 2 midpoint diagonal quadrilateral,
then (\ref{36}) holds. The slopes of the diagonals are $m_{1}=\dfrac{t}{s}$
and $m_{2}=\dfrac{w-u}{v}=\allowbreak \dfrac{2w-t}{2v-s}$, which implies
that $m_{1}m_{2}+1=\dfrac{t}{s}\dfrac{2w-t}{2v-s}+1=\allowbreak \dfrac{%
2(vs+wt)-(s^{2}+t^{2})}{s\left( 2v-s\right) }=0$ by (\ref{36}).
\end{proof}

\section{Main Result}

\qquad\ The following lemma allows us to express the eccentricity of an
ellipse as a function of the coefficients of an equation of that ellipse.

\begin{lemma}
\label{L4}Suppose that $E_{0}$ is an ellipse with equation $%
Ax^{2}+Bxy+Cy^{2}+Dx+Ey+F=0$; Let $a$ and $b$ denote the lengths of the
semi--major and semi--minor axes, respectively, of $E_{0}$. Then 
\begin{equation}
\dfrac{b^{2}}{a^{2}}=\dfrac{A+C-\sqrt{(A-C)^{2}+B^{2}}}{A+C+\sqrt{%
(A-C)^{2}+B^{2}}}\text{.}  \label{eccab}
\end{equation}
\end{lemma}

\begin{proof}
By \cite{S}, 
\begin{eqnarray}
a^{2} &=&\delta \dfrac{A+C+\sqrt{(A-C)^{2}+B^{2}}}{2}  \label{absq} \\
b^{2} &=&\delta \dfrac{A+C-\sqrt{(A-C)^{2}+B^{2}}}{2}\text{,}  \notag
\end{eqnarray}

where $\delta =4\dfrac{(CD^{2}+AE^{2}-BDE)-F(4AC-B^{2})}{(4AC-B^{2})^{2}}$;
Note that $\delta >0$ by Lemma \ref{L1}. (\ref{eccab}) then follows
immediately from (\ref{absq}).
\end{proof}

We now state and prove our main result, which gives a geometric
characterization of the unique ellipse of minimal eccentricity inscribed in
a midpoint diagonal quadrilateral.

\begin{theorem}
\label{T1}

(i) There is a unique ellipse of minimal eccentricity, $E_{I}$, \textit{%
inscribed} in a midpoint diagonal quadrilateral, $Q$.

(ii) Furthermore, the smallest non--negative angle between equal conjugate
diameters of $E_{I}$ equals the smallest non--negative angle between the
diagonals of $Q$.
\end{theorem}

\begin{remark}
In \cite{H1} we proved that there is a unique ellipse of minimal
eccentricity \textit{inscribed} in any convex quadrilateral, $Q$; The
uniqueness for midpoint diagonal quadrilaterals would then follow from that
result. However, the proof here, specialized for midpoint diagonal
quadrilaterals, is self--contained, uses different methods, and does not
require the result from \cite{H1}.
\end{remark}

\begin{proof}
If $Q$ is a parallelogram, then Theorem \ref{T1} was proven in \cite{H2}.
Now suppose that $Q$\ is \textbf{not }a parallelogram. Then by Lemma \ref{L2}%
, $Q$\ is \textbf{not }a trapezoid. Thus by using an isometry of the plane,
we may assume that $Q$\ has vertices $(0,0),(0,u),(s,t)$, and $(v,w)$, where 
$s,t,u,v$, and $w$ satisfy (\ref{R0}), (\ref{R1}), and (\ref{R2}). If $E_{0}$
is an ellipse inscribed in $Q$, then by Proposition \ref{P2}, the equation
of $E_{0}$ is $A(h)x^{2}+B(h)xy+C(h)y^{2}+D(h)x+E(h)y+F(h)=0$ for some $h\in
I$, where $A(h)$ thru $F(h)$ are given by (\ref{2}). Let $a=a(h)$ and $%
b=b(h) $ denote the lengths of the semi--major and semi--minor axes,
respectively, of $E_{0}$. Since the square of the eccentricity of $E_{0}$
equals $1-\dfrac{b^{2}}{a^{2}}$, it suffices to maximize $\dfrac{b^{2}}{a^{2}%
}$, which is really a function of $h\in I\ $since we allow $E_{0}$\ to vary
over all ellipses inscribed in $Q$; Thus we want to maximize $G(h),h\in I$,
where $G(h)=\dfrac{b^{2}(h)}{a^{2}(h)}$; Using (\ref{eccab}), $G(h)=\dfrac{%
A(h)+C(h)-\sqrt{{\large (}A(h)-C(h){\large )}^{2}+{\large (}B(h){\large )}%
^{2}}}{A(h)+C(h)+\sqrt{{\large (}A(h)-C(h){\large )}^{2}+{\large (}B(h)%
{\large )}^{2}}}=\dfrac{J(h)-\sqrt{M(h)}}{J(h)+\sqrt{M(h)}}$, where $J$ and $%
M$ are given by (\ref{jm}). By Lemma \ref{L5}, 
\begin{equation}
G^{\prime }(h)=\dfrac{p(h)}{\sqrt{M(h)}{\large (}J(h)+\sqrt{M(h)}{\large )}%
^{2}}\text{,}  \label{dG}
\end{equation}

where $p$ is the quartic polynomial given by 
\begin{equation}
p(h)=2J^{\prime }(h)M(h)-J(h)M^{\prime }(h)\text{.}  \label{p}
\end{equation}

\qquad We now prove Theorem \ref{T1} for the case when $Q$ is a type 1
midpoint diagonal quadrilaterals--the proof for Type 2 midpoint diagonal
quadrilaterals being similar. So we now adapt the formulas in (\ref{2}) and
in (\ref{eccab}) to type 1 midpoint diagonal quadrilaterals.

\textbf{Uniqueness:} First we show that there is a unique ellipse of minimal
eccentricity, $E_{I}$, inscribed in $Q$. Then we shall prove the property
about the angle between equal conjugate diameters of $E_{I}$. As earlier,
let $L$ be the line thru the midpoints of the diagonals of $Q$, so that the
equation of $L$ is given by (\ref{9}); Let $m_{1}=\dfrac{t}{s}$ and $m_{2}=%
\dfrac{w-u}{v}$ denote the slopes of $D_{1}$ and $D_{2}$, respectively; By
Lemma \ref{L14}, $Q$ is Type 1 if and only if (\ref{u1}) holds. Substituting
for $u$ in (\ref{p}) using (\ref{u1}) yields $p(h)=\allowbreak 256h\left( 
\dfrac{s-v}{s}\right) ^{4}\left( vt-ws\right) ^{2}\left( s-h\right) o(h)$,
where\ $o$ is given by (\ref{oK}). Note that Lemma \ref{L11} implies that 
\begin{equation}
J(h)+\sqrt{M(h)}>0,h\in I\text{.}  \label{37}
\end{equation}

Now assume first that $Q$ is a tangential quadrilateral. Then $Q$ is an
orthodiagonal quadrilateral by Lemma\textbf{\ \ref{L13}}, and so the
diagonals of $Q$ are perpendicular. Also, there is a unique circle, $\Phi $,
inscribed in $Q$, which implies that $\Phi $ is the unique ellipse of
minimal eccentricity inscribed in $Q$ since $\Phi $ has eccentricity $0$.
Since any pair of perpendicular diameters of a circle are equal conjugate
diameters, the smallest non--negative angle between conjugate diameters of a
circle is $\dfrac{\pi }{2}$. Hence Theorem \ref{T1} holds when $Q$ is a
tangential quadrilateral. So assume now that $Q$ is \textbf{not }a
tangential quadrilateral, which implies that $A(h)\neq C(h)$ for all $h\in I$
and thus $M(h)\neq 0$ for all $h\in I$ by (\ref{jm}); Since $M$ is
non--negative we have 
\begin{equation}
M(h)>0,h\in I\text{.}  \label{31}
\end{equation}

By (\ref{dG}), (\ref{37}), and (\ref{31}), $G$ is differentiable on $I$. By (%
\ref{dG}) and (\ref{p}), 
\begin{equation}
G^{\prime }(h)=\dfrac{256h\left( \dfrac{s-v}{s}\right) ^{4}\left(
vt-ws\right) ^{2}\left( s-h\right) o(h)}{\sqrt{M(h)}{\large (}J(h)+\sqrt{M(h)%
}{\large )}^{2}}\text{.}  \label{dGo}
\end{equation}

By Lemma \ref{L10} and (\ref{dGo}), $h_{+}$ is the unique root of $G^{\prime
}$ in $I$, where $h_{+}$ is given by (\ref{oroots}). Since $G(h)=\dfrac{%
b^{2}(h)}{a^{2}(h)}$, it follows that $G(h)>0$ on $I$; Also, $G\left( \dfrac{%
v}{2}\right) =G\left( \dfrac{s}{2}\right) =0$; Since $G$ is positive in the
interior of $I$ and vanishes at the endpoints of $I$, $h_{+}$ must yield the
global maximum of $G$ on $I$. That proves uniqueness.

\textbf{Angle between equal conjugate diameters: }We now prove that the
smallest non--negative angle between equal conjugate diameters of $E_{I}$
equals the smallest non--negative angle between the diagonals of $Q$. First
we find a simplified formula for $G(h_{+})$; Solving for $h_{+}^{2}$ in the
equation $o(h_{+})=0$ yields%
\begin{equation}
h_{+}^{2}=\dfrac{(s-2h_{0})K}{2\left( s^{2}+t^{2}\right) \left( s-v\right) }%
\text{,}  \label{h0sq}
\end{equation}%
where $K$ is given by (\ref{oK}). Substituting for $u$ again in the formulas
for $A(h)$ and $B(h)$ from (\ref{2}) using (\ref{u1}) and simplifying gives 
\begin{gather}
A(h_{+})+C(h_{+})=4\left( s-v\right) \times  \label{ApC} \\
\left( \dfrac{\left( s-v\right) (t^{2}+s^{2})}{s^{2}}h_{+}^{2}+\dfrac{%
2w\left( vt-ws\right) }{s}h_{+}-\allowbreak w\left( vt-ws\right) \right) 
\notag
\end{gather}%
and 
\begin{gather}
A(h_{+})-C(h_{+})=4\left( s-v\right) \times  \label{AmC} \\
\left( \dfrac{K(t^{2}-s^{2})(s-2h_{+})}{2s^{2}\left( s^{2}+t^{2}\right) }+%
\dfrac{2w\left( vt-ws\right) }{s}h_{+}-\allowbreak w\left( vt-ws\right)
\right) \text{.}  \notag
\end{gather}%
Using (\ref{ApC}) and (\ref{h0sq}) and simplifying then gives 
\begin{equation}
J(h_{+})=\dfrac{2p_{1}\left( s-v\right) }{s^{2}}(s-2h_{+})\text{.}
\label{Jh0}
\end{equation}%
(\ref{h0sq}) also yields 
\begin{equation}
B(h_{+})=-\dfrac{4\left( s-v\right) w{\large (}2wst-(t^{2}-s^{2})v{\large )}%
}{s^{2}+t^{2}}{\large (}s-2h_{+})\text{.}  \label{Bh0}
\end{equation}%
Using (\ref{Bh0}), (\ref{h0sq}), and (\ref{AmC}) gives 
\begin{gather}
M(h_{+})=\dfrac{4p_{1}\left( s-v\right) ^{2}}{(s^{2}+t^{2})s^{4}}\times
\label{Mh0} \\
(s-2h_{+})^{2}\allowbreak {\large (}2wst-(t^{2}-s^{2})v{\large )}^{2}\text{,}
\notag
\end{gather}%
${\large \ }$where $M$ is given by (\ref{jm}). Note that $%
2wst-(t^{2}-s^{2})v\neq 0$ since $M(h_{+}\allowbreak )>0$ by (\ref{31});
Also, $h_{+}\in I$ implies that $(s-v)(2h_{+}-s)<0$; Thus (\ref{Mh0}) yields 
\begin{equation}
\sqrt{M(h_{+})}=\dfrac{2\sqrt{p_{1}}\left( s-v\right) \left\vert
2wst-(t^{2}-s^{2})v\right\vert (s-2h_{+})}{s^{2}\sqrt{s^{2}+t^{2}}}\text{.}%
\allowbreak  \label{sqMh0}
\end{equation}%
By (\ref{Jh0}) and (\ref{sqMh0}) we have 
\begin{gather}
\left( J(h_{+}\allowbreak )+\sqrt{M\left( h_{+}\right) }\right) ^{2}=\dfrac{%
4p_{1}\left( s-v\right) ^{2}(s-2h_{+})^{2}}{s^{4}}\times  \label{24} \\
\left( \sqrt{p_{1}}+\dfrac{\allowbreak \left\vert
2wst-(t^{2}-s^{2})v\right\vert }{\sqrt{s^{2}+t^{2}}}\right) ^{2}\text{.} 
\notag
\end{gather}%
(\ref{Jh0}) and (\ref{Mh0}) imply that 
\begin{gather}
J^{2}(h_{+}\allowbreak )-M\left( h_{+}\right) =\dfrac{4p_{1}\left(
s-v\right) ^{2}(s-2h_{+})^{2}}{s^{4}}\times  \label{25} \\
\left( p_{1}-\dfrac{{\large (}2wst-(t^{2}-s^{2})v{\large )}^{2}}{s^{2}+t^{2}}%
\right) \text{.}  \notag
\end{gather}%
By (\ref{24}) and (\ref{25}) we have $G\left( h_{+}\right) \allowbreak =%
\dfrac{J^{2}(h_{+}\allowbreak )-M\left( h_{+}\right) }{{\large (}%
J(h_{+}\allowbreak )+\sqrt{M\left( h_{+}\right) }{\large )}^{2}}=$

$\dfrac{p_{1}-\tfrac{{\large (}2wst-(t^{2}-s^{2})v{\large )}^{2}}{s^{2}+t^{2}%
}}{\left( \sqrt{p_{1}}+\tfrac{\allowbreak \left\vert
2wst-(t^{2}-s^{2})v\right\vert }{\sqrt{s^{2}+t^{2}}}\right) ^{2}}$, which
implies that 
\begin{gather}
G\left( h_{+}\right) \allowbreak =  \label{Ghp} \\
\dfrac{\sqrt{s^{2}+t^{2}}\sqrt{p_{1}}-\allowbreak \left\vert
2wst-(t^{2}-s^{2})v\right\vert }{\sqrt{s^{2}+t^{2}}\sqrt{p_{1}}+\allowbreak
\left\vert 2wst-(t^{2}-s^{2})v\right\vert }\text{.}  \notag
\end{gather}%
$\allowbreak $

Now let $\alpha $ denote the smallest non--negative angle between the
diagonals of $Q$, so that $0\leq \alpha \leq \dfrac{\pi }{2}$, and let $%
m_{1}=$ $\dfrac{t}{s}$ and $m_{2}=\dfrac{w-u}{v}$ denote the slopes of the
diagonals of $Q$; Substituting for $u$ using (\ref{u1}) yields $\dfrac{%
m_{2}-m_{1}}{1+m_{1}m_{2}}=\dfrac{\allowbreak 2s\left( vt-ws\right) }{%
(t^{2}-s^{2})v-2wts}$; Using the formula $\tan \alpha =\left\vert \dfrac{%
m_{2}-m_{1}}{1+m_{1}m_{2}}\right\vert $ then yields, by (\ref{R1}), 
\begin{equation}
\tan \alpha =\dfrac{\allowbreak 2s\left( vt-ws\right) }{\left\vert
(t^{2}-s^{2})v-2wts\right\vert }\text{.}  \label{34}
\end{equation}

Now let $\tau _{1}$ and $\tau _{2}$ denote a pair of \textit{equal}
conjugate diameters of any ellipse, $E_{0}$; Let $\theta $ denote the acute
angle going counterclockwise from the major axis of $E_{0}$ to one of the
equal conjugate diameters, and let $\Gamma =$ angle between the \textit{equal%
} conjugate diameters of $E_{0},0\leq \Gamma \leq \pi $; If $a$ and $b$ are
the lengths of the semi--major and semi--minor axes, respectively, of $E_{0}$%
, then it is known that $\tau _{1}$ and $\tau _{2}$ make \textit{equal acute
angles}, on \textit{opposite} sides, with the major axis of $E_{0}$; Thus $%
\Gamma =2\theta $ and $\tan \theta =\dfrac{b}{a}$(see page 170 of \cite{SA}%
). Note that Salmon refers to $\theta $ as the angle $\tau _{1}$ or $\tau
_{2}$ makes with the axis of $x$; But Salmon is assuming the ellipse has
major axis $=x$ axis; So $\theta $ is really the angle with the major axis
of $E_{0}$; Assume now for the rest of the proof that $E_{0}=E_{I}$, the
unique ellipse of minimal eccentricity inscribed in $Q$; We then want to
show that $\alpha =\Gamma $, which is equivalent to showing that $\tan
2\theta =\tan \alpha $, which in turn is equivalent to showing that $\tan
^{2}2\theta =\tan ^{2}\alpha $ since we are assuming that $\theta $ and $%
\alpha $ lie in the first quadrant. Now $\tan 2\theta =\dfrac{2\tan \theta }{%
1-\tan ^{2}\theta }=\dfrac{2\tfrac{b}{a}}{1-\tfrac{b^{2}}{a^{2}}}$, which
implies that $\tan ^{2}2\theta =\dfrac{4\left( \tfrac{b^{2}}{a^{2}}\right) }{%
\left( 1-\tfrac{b^{2}}{a^{2}}\right) ^{2}}$; As shown above in the first
part of the proof, $h_{+}$ must yield the global maximum of $G$ on $I$. $%
G(h_{+})=\dfrac{b^{2}}{a^{2}}$ implies that 
\begin{equation}
\tan ^{2}2\theta =\dfrac{4G(h_{+})}{{\large (}1-G(h_{+}){\large )}^{2}}\text{%
.}  \label{32}
\end{equation}

Thus we must show that 
\begin{equation}
\dfrac{4G(h_{+})}{{\large (}1-G(h_{+}){\large )}^{2}}=\left( \dfrac{2s\left(
vt-ws\right) }{(t^{2}-s^{2})v-2wts}\right) ^{2}\text{.}  \label{33}
\end{equation}

Using (\ref{Ghp}), $1-G\left( h_{+}\right) =\dfrac{2\allowbreak \left\vert
2wst-(t^{2}-s^{2})v\right\vert }{\sqrt{s^{2}+t^{2}}\sqrt{p_{1}}+\left\vert
2wst-(t^{2}-s^{2})v\right\vert }$, which implies that 
\begin{gather}
{\large (}1-G(h_{+}){\large )}^{2}=  \label{38} \\
\dfrac{4{\large (}2wst-(t^{2}-s^{2})v{\large )}^{2}}{{\large (}\sqrt{%
s^{2}+t^{2}}\sqrt{p_{1}}+\left\vert 2wst-(t^{2}-s^{2})v\right\vert {\large )}%
^{2}}\text{.}  \notag
\end{gather}%
(\ref{33}) then follows from (\ref{Ghp}) and (\ref{38}).
\end{proof}

\begin{remark}
We do not know if it is possible for the second part of Theorem \ref{T1} to
hold when $Q$ is not a midpoint diagonal quadrilateral. If there were a
quadrilateral, $Q$, such that $Q$ is both a tangential and an orthodiagonal
quadrilateral, but not a midpoint diagonal quadrilateral, then it would
follow easily that the second part of Theorem \ref{T1} holds. However, it is
not hard to show that if $Q$ is both a tangential and an orthodiagonal
quadrilateral, then $Q$ must be a midpoint diagonal quadrilateral.
\end{remark}

\begin{remark}
Suppose that $Q$ is a Type 2 midpoint diagonal quadrilateral. If one
reflects $Q$ thru the $y$ axis one obtains a Type 1 midpoint diagonal
quadrilateral. Thus it might appear that one need only prove Theorem \ref{T1}
for Type 1 midpoint diagonal quadrilaterals. However, in the proof above we
also assumed that $Q$ has vertices $(0,0),(0,u),(s,t)$, and $(v,w)$.
Reflection thru the $y$ axis does not preserve the form of those vertices.
\end{remark}

\section{Example\textbf{\ }}

Suppose that $Q$\ has vertices $(0,0),(0,u),(s,t)$, and $(v,w)$, where $s=4$%
, $t=6$, $v=2$, $w=1$, and $u=\allowbreak 2$; Then $s,t,u,v,w$ satisfy (\ref%
{R0}), (\ref{R1}), and (\ref{R2}), and $I=(1,2)$; $Q$ is a type 1 midpoint
diagonal quadrilateral since $u=\allowbreak \dfrac{vt-ws}{s}$; $\tan \mu =%
\dfrac{\allowbreak 2s\left( vt-ws\right) }{\left\vert
(t^{2}-s^{2})v-2wts\right\vert }=\allowbreak 8$; $o(h)=\allowbreak
16(-13h^{2}-18h+36)$, which has roots $\dfrac{3}{13}(-3\pm \sqrt{61})$; $%
h_{+}=\dfrac{3}{13}(-3+\sqrt{61})$ and $\dfrac{4G(h_{+})}{{\large (}%
1-G(h_{+}){\large )}^{2}}=\dfrac{4\left( \tfrac{33-\sqrt{65}}{32}\right) }{%
\left( 1-\tfrac{33-\sqrt{65}}{32}\right) ^{2}}=\allowbreak \allowbreak
64=\tan ^{2}\mu $; The equation of $E_{I}$ is 
\begin{gather*}
(35-3\sqrt{61}){\large (}29x^{2}-4xy+36y^{2}) \\
+48(\allowbreak 72-11\sqrt{61})(x+2y) \\
+16(887-105\sqrt{61})=0\text{.}
\end{gather*}


\begin{thebibliography}{9}
\bibitem{C} G. D. Chakerian, A Distorted View of Geometry, MAA, Mathematical
Plums, Washington, DC, 1979, 130--150.

\bibitem{D} Heinrich D\"{o}rrie: 100 Great Problems of Elementary
Mathematics, Dover, New York, 1965.

\bibitem{H1} Alan Horwitz, \textquotedblleft Ellipses of maximal area and of
minimal eccentricity inscribed in a convex quadrilateral\textquotedblright ,
Australian Journal of Mathematical Analysis and Applications, 2(2005), 1-12.

\bibitem{H2} Alan Horwitz, \textquotedblleft Ellipses Inscribed in
Parallelograms\textquotedblright , Australian Journal of Mathematical
Analysis and Applications, Volume 9, Issue 1(2012), 1--12.

\bibitem{S} Mohamed Ali Said, "Calibration of an Ellipse's Algebraic
Equation and Direct Determination of its Parameters", Acta Mathematica
Academiae Paedagogicae Ny regyh aziensis Vol.19, No. 2 (2003), 221--225.

\bibitem{SA} George Salmon, \textit{A treatise on conic sections, 6th edition%
}, Chelsea Publishing Company, New York.
\end{thebibliography}
\end{document}